\documentclass[11pt,a4paper,final]{article}

\usepackage{a4}
\usepackage{amsfonts}                  
\usepackage{amsmath}                   
\usepackage{amsthm}
\usepackage{color}
\usepackage{amssymb}
\usepackage{enumerate}
\usepackage{soul}

\usepackage{amssymb}
\usepackage{enumerate}
\usepackage{soul}
\usepackage{pgfplots}
\pgfplotsset{compat=1.15}
\usepackage{mathrsfs}
\usetikzlibrary{arrows}

\DeclareMathOperator{\eco}{econv}	

\DeclareMathOperator{\conv}{conv}

\DeclareMathOperator{\cl}{cl}

\DeclareMathOperator{\epco}{\normalfont{e}^\prime conv} 	
 	
\DeclareMathOperator{\ep}{\text{e}^\prime} 	
\DeclareMathOperator{\dom}{dom}					
\DeclareMathOperator{\epi}{epi}

\newcommand{\Ramp}{\overline{\mathbb{R}}}
\newcommand{\R}{\mathbb{R}}
\newcommand{\A}{\mathcal{A}}
\newcommand{\B}{\mathcal{B}}

\newcommand{\RTD}{\mathbb{R}^{(T)}}
\newcommand{\ci}{\left\langle}		
\newcommand{\cd}{\right\rangle}

\newcommand{\ECC}{\text{ECC}}
\newcommand{\ECCQ}{\text{ECCQ}}

\def\supp{\mathop{\rm sup}}

\newcommand{\xb}{\overline{x}}

\theoremstyle{plain}
\newtheorem{theorem}{Theorem}[section]
\newtheorem{lemma}[theorem]{Lemma}
\newtheorem{corollary}[theorem]{Corollary}
\newtheorem{proposition}[theorem]{Proposition}

\theoremstyle{definition}
\newtheorem{definition}[theorem]{Definition}
\newtheorem{example}[theorem]{Example}

\theoremstyle{remark}
\newtheorem{remark}{Remark}

\begin{document}

\title{A new insight into Lagrange duality on DC optimization}

\author{Maria Dolores Fajardo\thanks{Department of Mathematics, University of Alicante, Alicante, Spain, md.fajardo@ua.es} \and
Jos\'e Vidal\thanks{Department of Physics and Mathematics, University of Alcalá, Alcalá de Henares, Spain, j.vidal@uah.es}}

\maketitle

\begin{abstract}
In this paper we present a new Lagrange dual problem associated to a primal DC optimization problem under the additivity condition (AC). As usual for DC programming, even weak duality is not guaranteed for free and, due to this issue, we investigate conditions not only for weak, but also for strong duality between this dual and the primal DC problem. In addition, we also analyze conditions for strong duality between the primal and its standard Lagrange dual problem utilizing a revisited regularity condition which is guaranteed by a more general class of functions than the one used in a prior work (see \cite{FVR2016}) .
\end{abstract}

\textbf{Keywords:~}{Generalized convex conjugation, evenly convex function, DC problem, Lagrange duality, weak duality, strong duality, regularity condition}\\

\textbf{MSC Subject Classification.} 52A20, 26B25, 90C25, 49N15.\\

\section{Introduction}
\label{sec:Intro}

An optimization problem is called a DC optimization problem when its objective function is, maybe, non-convex but the difference of convex functions. Though it has a wide range of applications, see for instance the recent surveys \cite{TD2018,Oli2020} or \cite{TD2023}, it has not been as deeply studied as the case of convex programming. For an algorithmic point of view and diverse numeric applications, we refer the reader to \cite{WCP2018,TZGF2020,LHT2018,TFT2022,LLD2007}. 
For a more theoretical perspective, some interesting references on the basis of this topic are \cite{TD2023,AT2005,TA1997,D2020,D2022,T1995} or \cite{HT1999}. 

As shown in \cite{MLV1999} or \cite{HU1986}, it is possible to address these problems via a duality approach, opening in this way the use of the Fenchel conjugate to derive appropriate dual problems. These dual problems enrolled to the original problem of interest, called \textit{primal}, are said to satisfy weak duality when the optimal value of the latter is bigger or equal than the one of the former, zero duality gap when both values coincide and strong duality when they are equal and, additionally, the dual problem in solvable. For an outstanding overview of conjugate duality theory in different spaces, we refer the reader to \cite{B2010,R1970,BC2011,Z2002} or the book \cite{MC2022}. 

In a recent work, \cite{FV2023}, the authors of this manuscript provided a Fenchel duality theory associated with a DC problem, on the basis of the previous paper   \cite{FVR2012}, which uses the perturbational approach from \cite{ET1976} with a generalized conjugation scheme, the $c$-conjugation, defined in \cite{MLVP2011} and inspired by a survey from Martínez-Legaz, \cite{ML2005}, concerning quasiconvex programming. This conjugation pattern is linked to a class of convex functions called evenly convex functions, which generalizes the one of closed and convex functions, and are characterized  (see \cite{RVP2011}) as those functions whose epigraph is the intersection of an arbitrary family of open half spaces, i.e., it is an evenly convex set. Going backwards in time, these sets were introduced by Fenchel in \cite{F1952} in the finite dimensional case, who defined a natural polarity operation for them, preserving important properties  exhibited by the standard polarity related to closed convex sets. For almost thirty years, these sets appeared occasionally, but after 1980, they were used broadly in quasiconvex programming, see for instance \cite{ML1983} and \cite{PP1984}. Later, evenly convex sets have implemented the geometric study of linear inequality systems with strict and weak inequalities \cite{GJR2003,GR2006} and have been studied in terms of their sections and projections in \cite{KMZ2007}.

Coming back to the developed work in \cite{FVR2012}, a kind of  biconjugation theorem, which characterizes proper evenly convex functions as those proper functions which coincide with its biconjugate, allowed to obtain strong  duality sufficient conditions if the involved functions in the primal problem are evenly convex. Duality theory in optimization problems with  either evenly convex objective functions and feasible sets, or whose developed conditions for strong duality are in terms of properties related to the even convexity of certain sets, called evenly convex optimization, has been deeply studied in the last years. For a comprehensive reading in this topic, see \cite{FGRVP2020} and the references therein.

The intention of this work is to continue from another perspective two previous works by the authors on Lagrange duality in evenly convex optimization. On the one hand, in  \cite{FVR2016}, we developed a Lagrange dual problem and sufficient  conditions for strong duality, which are called regularity conditions. However, as it is indicated in that work, both the objective function and the constraints are supposed to be evenly convex functions instead of being just convex. On the other hand and following this line of generalism, one of the scopes of the recent work \cite{FV2024} goes in that regard. In that work we investigate not only a general primal DC problem, but also we recover the case of \cite{FVR2016} and characterize strong duality and zero duality gap between the primal and its standard Lagrange dual. This characterization is done supposing the involved functions to be proper and just convex and it is done via characteristic sets. The achievements from \cite{FV2024} have motivated us to continue persuading sufficient conditions relaxing those properties somehow imposed in the previous scenario stated in \cite{FVR2016}. This explanation serves as motivation for Section \ref{sec:SD_DL_standard}.

Another significant goal of \cite{FV2024} was the development of a second Lagrange dual for the DC primal problem, on the basis of the standard Lagrange dual, using intrinsic properties of the $c$-conjugation scheme, being both dual problems equivalent whenever an even convexity requirement holds. The structure of this new dual has inspired us in the obtainment of a third Lagrange dual problem. This new dual will be detailed in Section \ref{sec:Dual_problems}.

The rest of the paper is organized as follows. Section \ref{sec:Preliminaries} contains all the necessary results used in the paper. Section \ref{sec:Dual_problems} develops a new Lagrange dual problem which is intrinsically related to the one obtained in \cite{FV2024}. In Section \ref{sec:SD_DL_standard} we obtain a new regularity condition for a general primal and its standard Lagrange dual problem where the involved functions are simply proper convex functions. The purpose of Section \ref{sec:SD_DL_tilde} is to characterize strong Lagrange duality for a general DC primal and the new Lagrange problem developed in Section \ref{sec:Dual_problems}. Finally, in Section \ref{sec:Conclusions} we summarize the main achievements of this manuscript leaving some open ideas for future research.

\section{Preliminaries}
\label{sec:Preliminaries}
Throughout the paper, $X$ represents a non-trivial separated locally convex space whose topology is the one induced by its continuous dual space $X^*$. We denote by $\ci x,x^*\cd$ the value of the continuous linear functional $x^* \in X^*$ at $x \in X$. For a subset $D$ of $X$, we denote by $\conv D$ and $\cl D$ its convex hull and closure, respectively. The indicator function of $D$, $\delta_D:X \rightarrow \Ramp:=\R\cup\left\{\pm\infty\right\}$, is denoted by
\begin{equation*}
\delta_D\left( x\right) =\left\{ 
\begin{array}{ll}
1, &~~~ \text{if }x\in D , \\ 
+\infty, &~~~ \text{otherwise.}
\end{array}
\right. 
\end{equation*} 
We use the notation $\R_+=[0,+\infty[$ and $\R_{++}=]0,+\infty[$.

Evenly convex sets, briefly e-convex sets, present the following characterization which is far more manageable than their initial definition.

\begin{definition}[Def.~1, \cite{DML2002}]
\label{def:Econvex}
A set $C\subseteq X$ is e-convex if for every point $x_0\notin C$, there exists $x^*\in X^*$ such that $\ci x-x_0,x^*\cd<0$, for all $x\in C$.
\end{definition}
If $D$ is a subset of $X$, its \textit{e-convex hull}, denoted by $\eco D$, is the smallest e-convex set containing $D$ which, moreover, verifies
\begin{equation*}
	D\subseteq \eco D\subseteq \cl \conv D.
\end{equation*}
Clearly, the e-convex hull operator is closed under arbitrary intersections and, thanks to Hahn-Banach theorem, every closed or open convex set is also e-convex. The converse statement is false in general; see \cite[Ex.~2.2]{FV2023}.

For a function $f:X\to\Ramp$, we denote by $\dom f:=\left\{ x\in X~:~f(x)<+\infty\right\}$ and $\epi f:=\left\{(x,\alpha)\in X\times\R~:~f(x)\leq\alpha\right\}$ its effective domain and its epigraph, respectively. It is \textit{proper} if $f(x)> -\infty$, for all $x \in X$, and its effective domain is non-empty. Its \textit{lower semicontinuous hull}, denoted by $\cl f$, is the function verifying $\epi \cl f =\cl \epi f$, and $f$ is called \textit{lower semicontinuous} if $f=\cl f$.
In a bit different way, the \textit{e-convex hull} of a proper function $f$, briefly denoted by $\eco f$, is its largest e-convex minorant. Let us observe that it is not defined as that function verifying $\epi\eco f=\eco\epi f$; see \cite[Ex.~1]{FV2024}. The class of e-convex functions contains strictly the class of lower semicontinuous convex functions, see \cite[Ex.~2.1]{FV2017}, due to the strict inclusion relation between e-convex sets and closed convex sets.
The following definition is taken from \cite{MLVP2011}.
\begin{definition}[Def.~15, \cite{MLVP2011}]\label{definition:eaffine}
A function $a:X\rightarrow \overline{\mathbb{R}}$ is said to be \emph{e-affine} if there exist $x^{\ast },y^{\ast }\in X^{\ast }$ and $\alpha ,\beta \in \mathbb{R}$ such that 
\begin{equation*}\label{equation:general_eaffine_function}
a\left( x\right) =\left\{ 
\begin{array}{ll}
\left\langle x,x^{\ast }\right\rangle -\beta, &~~~ \text{if }\left\langle
x,y^{\ast }\right\rangle <\alpha , \\ 
+\infty, &~~~ \text{otherwise.}
\end{array}
\right. 
\end{equation*}
\end{definition}
For any function $f:X\rightarrow \overline{\mathbb{R}}$, $\mathcal{E}_{f}$ denotes the set of all e-affine functions minorizing $f$, that is, $\mathcal{E}_{f}:=\left\{ a \, : \, X\rightarrow \overline{\mathbb{R}} \, , \, a\text{ is e-affine and }a\leq f\right\}$. From \cite{MLVP2011}, we have the next characterization for proper e-convex functions.

\begin{theorem}[Th.~16, \cite{MLVP2011}]\label{theorem:charpef}
Let $f:X\rightarrow \overline{\mathbb{R}}$ such that $f$ is not identically $+\infty $ or $-\infty$. Then, $f$ is
a proper e-convex function if and only if $f=\supp \left\{ a \, : \, a\in \mathcal{E}_{f}\right\}$.  
\end{theorem}

Due to the fact that the sum of e-affine functions is not e-affine, throughout this paper we will impose a technical assumption, expressed in terms of the following set whose definition comes from \cite{FVR2012}.

\begin{definition} \label{definition:Set Efg tilde}
Consider two functions $f_1,f_2:X\rightarrow \overline{\mathbb{R}}$. An e-affine function $a:X\rightarrow \overline{\mathbb{R}}$ belongs to the set $\widetilde{\mathcal{E}}_{f_1,f_2}$ if there exist $a_{1}\in \mathcal{E}_{f_1}$, $a_{2}\in \mathcal{E}_{f_2}$ such that, denoting 
\begin{equation*}
	a_{i}(\cdot )=\left\{ 
	\begin{array}{ll}
		\left\langle \cdot ,x_{i}^{\ast }\right\rangle -\beta _{i}, &~~~ \text{if } \left\langle \cdot ,y_{i}^{\ast }\right\rangle <\alpha _{i}, \\ 
		+\infty, &~~~ \text{otherwise,}
	\end{array}
	\right.
\end{equation*}
\noindent
for $i=1,2$, then 
\begin{equation*}
a(\cdot) =\left\{ 
	\begin{array}{ll}
		\ci \cdot ,x_1^{\ast }+x_2^{\ast }\cd - (\beta_1+\beta_2), &~~~ \text{if }\ci \cdot ,y_1^{\ast }+y_2^{\ast }\cd < \alpha_1+\alpha_2, \\
		+\infty, &~~~ \text{otherwise.}
	\end{array}
	\right.
\end{equation*}
\end{definition}

\begin{definition}
Let $f_1,f_2: X \rightarrow \Ramp$ be proper functions. We say that the \textit{additivity condition} (AC) holds for $f_1$ and $f_2$ if 
\begin{equation*}
	\tag{AC}
	\eco (f_1+ f_2)   =\sup\left\{ a\,:\,a\in\widetilde{\mathcal{E}}_{f_1,f_2}\right\}.
\end{equation*}
\end{definition}

Given a proper function $f:X\to\Ramp$, let us recall its Fenchel conjugate, $f^*:X^* \to \Ramp$, defined as
\begin{equation*}
	f^*(x^*):=\sup_{x \in X}\left\{  \ci x,x^*\cd-f(x)\right\}.
\end{equation*} 
There exists an equivalence between $f$ to be lower semicontinuous and convex, and the equality $f = f^{**}$ (Fenchel biconjugation theorem). In general, this result is not true for e-convex functions; see \cite[Ex.~2.5]{FV2016SSD}. The conjugation scheme that we will use throughout the paper, which is called the $c$-conjugation scheme, is based on the generalized convex conjugation theory presented by Moreau in \cite{Mor1970} and developed in \cite{MLVP2011}. Denoting by $W:=X^* \times X^* \times \R$, it uses the coupling function $c:X\times W\to\Ramp$
\begin{equation*}
	c(x,(x^*,y^*,\alpha))=\left\{
	\begin{array}{ll}
		\ci x,x^*\cd, &~~~ \text{if } \ci x,y^*\cd<\alpha,\\
		+\infty, & ~~~ \text{otherwise,}
	\end{array}
	\right.
\end{equation*}
and, for a proper function $f:X\to\Ramp$, its \emph{$c$-conjugate} $f^c:W \rightarrow \Ramp$ is defined as
\begin{equation*}
	f^c(x^*,y^*,\alpha):=\sup_{x \in X}\left\{ c(x,(x^*,y^*,\alpha))-f(x)\right\}.
\end{equation*}
Let us observe that, if  $H_{y^*,\alpha}^{-}:=\{x\in X : \ci x,y^*\cd < \alpha \}$, then
\begin{equation*}
	f^c(x^*,y^*,\alpha)= \left \{
	\begin{array}{ll}
		f^*(x^*), &~~~ \text{if } \dom f \subseteq H_{y^*,\alpha}^{-},\\
		+\infty, & ~~~ \text{otherwise.}
	\end{array}
	\right.
\end{equation*}
This conjugation scheme is complemented by the use of an additional coupling function $c^\prime:W\times X\to\Ramp$ defined as
\begin{equation*}
	c^\prime((x^*,y^*,\alpha),x):=c(x,(x^*,y^*,\alpha)),
\end{equation*}
which, for a proper function $h:W\to\Ramp$, allows to define its \emph{$c^\prime$-conjugate} function, $h^{c^\prime}:X \rightarrow \Ramp$, as follows
\begin{equation*}
	h^{c^\prime}(x):=\sup_{(x^*,y^*,\alpha)\in W}\left\{c^{\prime}((x^*,y^*,\alpha),x)-h(x^*,y^*,\alpha)\right\}.
\end{equation*}
This conjugation pattern gives priority to $-\infty$, i.e., the sign convention is 
\begin{equation*}
	(+\infty)+(-\infty)=(-\infty)+(+\infty)=(+\infty)-(+\infty)=(-\infty)-(-\infty)=-\infty.
\end{equation*}
The e-affine functions $c(\cdot,(x^*,y^*,\alpha))-\beta :X \rightarrow \Ramp$, with $(x^*,y^*,\alpha) \in W $ and $\beta \in \R$, are called \emph{c-elementary} in generalized convex duality theory (see \cite{Mor1970}) and functions $c^{\prime}(\cdot,x)-\beta : W \rightarrow \Ramp$, with $x\in X$ and $\beta \in \mathbb{R}$, are called \emph{c}$^{\prime }$-\emph{elementary}. Rewriting Theorem \ref{theorem:charpef}, it holds that any proper e-convex function $f: X \rightarrow \Ramp$ is the pointwise supremum of a set of $c$-elementary functions. Extending this concept, \cite{FVR2012} introduced the \emph{$\ep$-convex functions} as convex functions $g:W \rightarrow \Ramp$ which can be expressed as the pointwise supremum of a set of $c^\prime$-elementary functions. The \emph{$\ep$-convex hull} of any function $g:W\rightarrow\Ramp$, $\epco g$, is defined as the largest $\ep$-convex minorant function  of $g$. The next theorem from \cite{ML2005} is the counterpart of Fenchel-Moreau theorem for e-convex and $\ep$-convex functions.
 
\begin{theorem}[Prop. 6.1, 6.2, Cor. 6.1,~\cite{ML2005}]
\label{thm:Theorem_charac}
Let $f:X\rightarrow \overline{\mathbb{R}}$ and $g:W\rightarrow \overline{\mathbb{R}}$. The following statements are true.
\begin{itemize}
	\item[i)] $f^{c}$ is e$^{\prime }$-convex; $g^{c^{\prime }}$ is e-convex.
	\item[ii)] If $f$ has a proper e-convex minorant, then $\eco f=f^{cc^{\prime }}$; $ \epco g=g^{c^{\prime }c}$.
	\item[iii)] If $f$ does not take on the value $-\infty$, then $f$ is e-convex if and only if $f=f^{cc^{\prime}}$; $g$ is e$^{\prime }$-convex if and only if $g=g^{c^{\prime }c}$.
	\item[iv)] $f^{cc^{\prime }}\leq f$; $g^{c^{\prime }c}\leq g$. 
\end{itemize}
\end{theorem}

As a consequence of the previous result, we present the following corollary.

\begin{corollary} \label{cor:Corollaryconseq}
Let $f:X\rightarrow \overline{\mathbb{R}}$ and $g:W\rightarrow \overline{\mathbb{R}}$. Then, $(\eco f)^c=f^c$ and $(\epco g)^{c^{\prime }}=g^{c^{\prime }}$.
\end{corollary}

In the same natural manner than e-convex functions are connected to e-convex sets, $\ep$-convex sets arose from $\ep$-convex functions. Their main properties were presented in \cite{FV2020} and their role in duality results in evenly convex optimization is fundamental; see \cite{F2015,FVR2016,FV2018,FV2016SSD}.
\begin{definition}[Def.~2, \cite{FVR2012}]
A set $D \subset W \times \R$ is \textit{$\ep$-convex} if there exists an $\ep$-convex function $k : W \to \R$ such that $D = \epi k$. The \textit{$\ep$-convex hull} of an arbitrary set $D \subset W \times \R$ is defined as the smallest
$\ep$-convex set containing $D$ and it will be denoted by $\epco D$.
\end{definition}

Associated to the $c$-conjugation scheme is the notion of $c$-subdifferentiability of a function at a point, which was introduced in \cite{MLVP2011} and connected with evenly convex optimization in \cite{FVR2016}, \cite{FGV2021} and \cite{FV2022}.

\begin{definition}
Let $f:X \rightarrow \R$ and $\varepsilon \geq 0$. A point $(x^*,y^*,\alpha) \in W$ is an \textit{$\varepsilon$-$c$-subgradient} of $f$ at $\bar x \in \dom f$ if $\left\langle  \bar x, u^* \right\rangle < \alpha $ and, for all $x \in X$,
$$ f(x)-f(\bar x) \geq c (x,(x^*,y^*,\alpha))-c(\bar x,(x^*,y^*,\alpha))-\varepsilon. $$ 
In the case $\varepsilon=0$, these points are called $c$-subgradients.
The set of $\varepsilon$-$c$-subgradients ($c$-subgradients, resp.) of $f$ at $\bar x$ is denoted by $\partial_{c,\varepsilon}f(\bar x)$ ($\partial_{c}f(\bar x)$, resp.) and it is called the \textit{$\varepsilon$-$c$-subdifferential} (\textit{$c$-subdifferential}, resp.) of $f$ at $\bar x$. As usual, if $f(\bar x) \notin \R$, it yields $\partial_{c, \varepsilon}f(\bar x)=\emptyset$ and  $\partial_{c}f(\bar x)=\emptyset$.
\end{definition}
There exists a relevant connection between the $c$-subdifferentiability of a function and the $c$-subdifferentiability of its e-convex hull. The following result comes directly from Corollary \ref{cor:Corollaryconseq} and Theorem 4 ii) from \cite{FV2022}.
\begin{lemma} \label{lemma:c-sub_econv}
Let $f:X\rightarrow \Ramp$ be a proper function. Then 
$$\partial_{c,\varepsilon}f(x) \subseteq \partial_{c,\varepsilon}(\eco f)(x) $$
for all $\varepsilon \geq 0$ and $x \in X$.
\end{lemma}

Next lemma will be important in Section \ref{sec:SD_DL_tilde} due to the connection that establishes between the set $\epi f^c$ and the operator $\partial_{c,\varepsilon}f$.
\begin{lemma}[Lem.9,~\cite{FVR2012}]\label{Lema9_9}
Let $f:X\rightarrow \Ramp$ be a proper function. Then, if $x_0 \in \dom f$,
\begin{equation*}
	\epi f^c=\bigcup_{\varepsilon \geq 0} \left \{(x^*, y^*, \alpha, \left\langle x_0, x^*\right\rangle+\varepsilon-f(x_0):(x^*, y^*, \alpha) \in \partial_{c,\varepsilon} f(x_0)\right \}.
\end{equation*}
\end{lemma}

To conclude this section, we recall the following definition.

\begin{definition}
\label{def:Inf_conv}
Let $f_1,f_2: X \rightarrow \Ramp$ be proper functions. The \textit{infimal convolution function} of $f_1$ and $f_2$, $(f_1 \oplus f_2):X \to \Ramp$, is defined as
\begin{equation*}
	(f_1 \oplus f_2)(x)= \inf_{u\in X} \{ f_1(u)+f_2(x-u)\}.
\end{equation*}
Additionally, the infimal convolution is said to be \textit{exact at $x\in X$} when the equality $(f_1\oplus f_2)(x)=f_1(a)+f_2(x-a)$ holds for some $a\in X$.
\end{definition}

\section{Dual problems}
\label{sec:Dual_problems}
Let us consider a DC optimization problem
\begin{equation}
	\label{eq:Primal_problem}
	\tag{$P$}
	\inf_{x\in A} \left\{ f(x)-g(x) \right\},
	\end{equation}
where its feasible set is defined as
\begin{equation}
	\label{eq:Set_A}
	A=\left\{x\in X\,:\, h_t(x)\leq 0, ~t\in T\right\},
\end{equation}
being $f,g, h_t:X\to\Ramp$ proper convex functions, for all $t\in T$, with $T$ an arbitrary index set. In case $g\equiv 0$, we replace the notation of the primal problem \eqref{eq:Primal_problem} by $(P_0)$. We assume that $f(x)-g(x)=+\infty$ in case $x \not \in \dom f$, which implies that $f-g$ will be a proper function if and only if $\dom f \subseteq \dom g$, and, consequently, for all $(u^*, v^*, \gamma ) \in \dom g^c$, $\dom f \subseteq H_{v^*, \gamma}^-$.

Denoting by $\RTD$ the space of generalized finite sequences, i.e., a sequence $\lambda=(\lambda_t)_{t\in T}$ belongs to $\RTD$ if only finitely many $\lambda_t$ are different from zero, and representing by ${\RTD_+}$ its nonnegative polar cone, the product $\lambda h := \sum_{t \in T} \lambda_t h_t$, with $\lambda \in \RTD_+$ is well defined and used as duality product between these spaces. Associated to \eqref{eq:Primal_problem}, we recall the \textit{standard} Lagrange dual problem
\begin{equation}
	\label{eq:DL_standard}
	\tag{$D_L$}
	\sup_{\lambda\in\RTD_+}\inf_{x\in X} \left\{f(x)-g(x)+\lambda h(x)\right\},
\end{equation}
which was also obtained via perturbational approach and using $c$-conjugation instead of Fenchel conjugate pattern in \cite{FVR2016}, where it was deeply studied in case $g\equiv 0$ and sufficient conditions for strong  duality were obtained. Complementing this research, but from a completely different perspective, we characterized  in \cite{FV2024} not only strong duality, but also zero duality gap, when the primal is a general DC optimization problem, conditions which, of course, can be particularized for  $g\equiv 0$.\\ 
In case $g$ happens to be e-convex and as shown in \cite{FV2024}, the dual problem \eqref{eq:DL_standard} can be rewritten as 
\begin{equation}
	\label{eq:DL_bar}
	\tag{$\overline{D}_L$}
	\sup_{\lambda\in\RTD_+}\inf_{(x^*,y^*,\alpha)\in W} \left\{g^c(x^*,y^*,\alpha)-(f+\lambda h)^c(x^*,y^*,\alpha)\right\},
\end{equation}
but, in general, there is no relation between the optimal values of both Lagrange dual problems. However, if $f$ and $\lambda h$ satisfy condition (AC) for any $\lambda\in\RTD_+$, in virtue of \cite[Cor.~8]{FVR2012} and Corollary \ref
{cor:Corollaryconseq},  the function $f+\lambda h$ satisfies  

	$$(f+\lambda h)^c =\left(\eco (f+\lambda h)\right)^c =
	\epco (f^c \oplus (\lambda h)^c)\leq f^c \oplus (\lambda h)^c,$$
which, together with the structure of \eqref{eq:DL_bar}, allows us to suggest a new dual problem for $(P)$
\begin{equation}
	\label{eq:DL_tilde}
	\tag{$\widetilde{D}_L$}	
	\sup_{\lambda\in\RTD_+}\inf_{(x^*,y^*,\alpha)\in W} \left\{g^c(x^*,y^*,\alpha)-(f^c \oplus (\lambda h)^c)(x^*,y^*,\alpha)\right\}.
\end{equation}	
Then, if $f$ and $\lambda h$ satisfy condition (AC), for all $\lambda\in\RTD_+$, and $g$ is e-convex, we obtain the optimal values relation
\begin{equation}
	\label{eq:Relation_initial}
	v(P) \geq v\eqref{eq:DL_standard}=v\eqref{eq:DL_bar}\geq v\eqref{eq:DL_tilde}.
\end{equation}
The following example taken from \cite{FV2024} justifies that the e-convexity of $g$ is necessary for the validity of \eqref{eq:Relation_initial}.\\
\begin{example}
Let $T$ be a singleton and $f,g,h:\R\to\Ramp$ functions defined as
\begin{equation*}
	f(x)=\left\{
	\begin{array}{ll}
		x,&~~~\text{if } x\geq 0,\\
		+\infty, &~~~ \text{otherwise;}
	\end{array}
	\right.
	\hspace{0.35cm}
	g(x)=\left\{
	\begin{array}{ll}
		x,&~~~\text{if } x> 0,\\
		1,&~~~\text{if } x=0,\\
		+\infty, &~~~ \text{otherwise;}
	\end{array}
	\right.
	\hspace{0.35cm}
	h(x)=\left\{
	\begin{array}{ll}
		-x,&~~~\text{if } x\geq 0,\\
		+\infty, &~~~ \text{otherwise.}
	\end{array}
	\right.
\end{equation*}
We have $A=\left\lbrace x \in \R: h(x) \leq 0 \right\rbrace= \left[0, +\infty \right[$ and the optimal value of \eqref{eq:Primal_problem} is
\begin{equation*}
	v\eqref{eq:Primal_problem}=\inf\left\lbrace f(x)-g(x):x \geq 0 \right\rbrace=-1,
\end{equation*} 
and, for all $(x^*,y^*,\alpha) \in \R^3$, one obtains
\begin{align*}
g^c(x^*,y^*,\alpha) &=\left\{
	\begin{array}{ll}
		0,&~~~\text{if } x^* \leq 1, y^* \leq 0, \alpha >0,\\
		+\infty, &~~~ \text{otherwise.}
	\end{array}
	\right.
\end{align*}
According to the definition of $f$ and $h$, given $\lambda\in\R_+$, it follows
\begin{align*}
	f^c(x^*,y^*,\alpha)&=\left\{
	\begin{array}{ll}
		\sup_{x\geq 0}\left\{x(x^*-1)\right\}, &~~~ \text{if }y^*\leq 0,\,\alpha>0,\\
		+\infty, &~~~ \text{otherwise},
	\end{array}
	\right.\\
	&=\left\{
	\begin{array}{ll}
		0,&~~~ \text{if } x^*\leq 1,\,y^*\leq 0,\,\alpha>0,\\
		+\infty,&~~~ \text{otherwise},
	\end{array}
	\right.
\end{align*}
together with
\begin{align*}
	(\lambda h)^c(x^*,y^*,\alpha)&=\left\{
	\begin{array}{ll}
		\sup_{x\geq 0}\left\{x(x^*+\lambda)\right\}, &~~~ \text{if }y^*\leq 0,\,\alpha>0,\\
		+\infty, &~~~ \text{otherwise},
	\end{array}
	\right.\\
	&=\left\{
	\begin{array}{ll}
		0,&~~~ \text{if } x^*\leq -\lambda,\,y^*\leq 0,\,\alpha>0,\\
		+\infty,&~~~ \text{otherwise}.
	\end{array}
	\right.
\end{align*}
Then, for all $\lambda\geq 0$, 
\begin{align*}
(f^c\oplus(\lambda h)^c)(x^*,y^*,\alpha) &= 
	 \inf_{(u^*,v^*,\gamma)\in W}\left\{ f^c(u^*,v^*,\gamma)+(\lambda h)^c(x^*-u^*,y^*-v^*,\alpha-\gamma)\right\}\\[0.15cm]
&=\left\{
	\begin{array}{ll}
		0, &~~~ \text{if } x^*+\lambda \leq 1, y^*\leq 0,\,\alpha>0,\\
		+\infty, &~~~ \text{otherwise}.
	\end{array}
	\right.
\end{align*}
\noindent
Finally, we have
\begin{align*}
 \inf_{(x^*,y^*,\alpha)\in W} \left\{g^c(x^*,y^*,\alpha)-(f^c \oplus (\lambda h)^c)(x^*,y^*,\alpha)\right\}
= \left\{
	\begin{array}{ll}
		0, &~~~ \text{if } \lambda=0,\\
		-\infty, &~~~ \text{otherwise},
	\end{array}
	\right.
\end{align*}
\noindent
concluding that 
\begin{equation*}
	v\eqref{eq:DL_tilde}=\sup_{\lambda\geq 0}\inf_{(x^*,y^*,\alpha)\in W} \left\{g^c(x^*,y^*,\alpha)-(f^c \oplus (\lambda h)^c)(x^*,y^*,\alpha)\right\}=0,
\end{equation*}
so neither weak duality for \eqref{eq:Primal_problem}$-$\eqref{eq:DL_tilde} holds nor \eqref{eq:Relation_initial} remains true. \qed
\end{example}

In the light of this example and as it happened previously in \cite{FV2023} for Fenchel duality and in \cite{FV2024} for Lagrange duality in DC optimization, it is clear that not only sufficient conditions guaranteeing strong duality for the pair \eqref{eq:Primal_problem}$-$\eqref{eq:DL_tilde} are needed, but also conditions ensuring weak duality for this dual pair must be developed. 

Though the main objective of the rest of the paper is to study the primal-dual pair \eqref{eq:Primal_problem}$-$\eqref{eq:DL_tilde}, we will start this research approaching conditions for strong duality between \eqref{eq:Primal_problem}$-$\eqref{eq:DL_standard} since some results from the next section will be needed later on in the analysis of \eqref{eq:Primal_problem}$-$\eqref{eq:DL_tilde}.

\section{New duality results for $(P)-(D_L)$}
\label{sec:SD_DL_standard}
In this section we start revisiting a regularity condition for strong duality between $(P_0)-$\eqref{eq:DL_standard} developed in \cite{FVR2016}. This condition is related to the equality
\begin{equation}
	\label{eq:Prop4_1_FVR2015}
	\epi \delta_A^c = \epco \Bigg(\bigcup_{\lambda\in\RTD_+} \epi(\lambda h)^c\Bigg),
\end{equation}
which, due to \cite[Prop.~4.1]{FVR2016}, allows to show that the condition
\begin{equation}
\label{eq:ECCQ}
	\tag{$\ECCQ$}
	\epi \delta_A^c=\bigcup_{\lambda\in\RTD_+} \epi(\lambda h)^c 
\end{equation} 
which actually means that the set $\bigcup_{\lambda\in\RTD_+} \epi(\lambda h)^c$ is $\ep$-convex, together with additional properties, see \cite[Th.~4.1]{FVR2016}, become sufficient for strong duality for $(P_0)-({D}_L)$. As explained there, this condition turns out to be the e-convex counterpart of the \textit{closed
cone constraint qualification} (CCCQ) for strong Lagrange duality in \cite{JDL2004}, and the functions $f$ and $h_t$, for all $t \in T$ are assumed to be all e-convex. Next example shows that the e-convexity of the constraints is necessary for  \eqref{eq:Prop4_1_FVR2015}.

\begin{example}
\label{ex:Constraints_need_to_be_econvex}
Let $T$ be a singleton and $h:\R\to\Ramp$ be the function defined by
\begin{equation*}
	h(x)=\left\{
	\begin{array}{ll}
		x,&~~~ \text{if } x<0,\\
		1,&~~~ \text{if } x=0,\\
		+\infty,&~~~ \text{if } x>0.
	\end{array}
	\right.
\end{equation*}
It is clear that $A=]-\infty,0[$. Applying the definition of $c$-conjugate function, it is not complicated to see that 
\begin{equation*}
	\epi \delta_A^c=\R_+\times \big((\R_+\times\R_+)\backslash\left\{(0,0)\right\}\big)\times \R_+.
\end{equation*}
Now, for all $\lambda\geq 0$, it follows that 
\begin{equation*}
	\epi(\lambda h)^c =  [\lambda,+\infty[\times\big(\R_+\times\R_{++}\big)\times \R_+,
\end{equation*}
so we have
\begin{equation*}
	\bigcup_{\lambda\in\RTD_+}\epi(\lambda h)^c=\R_+\times\big(\R_+\times\R_{++}\big)\times\R_+.
\end{equation*}
Finally, we conclude that 
\begin{equation*}
	\epco \Bigg( \bigcup_{\lambda\in\RTD_+}\epi(\lambda h)^c \Bigg)\subsetneqq \epi\delta_A^c, 
\end{equation*}
which shows that the e-convexity of $h$ is necessary for the fulfillment of \eqref{eq:Prop4_1_FVR2015}.\qed
\end{example}
This example suggests to ensure the accomplishment of condition (ECCQ) without the e-convexity of the constraint functions determining the feasible set $A$. To this aim, we define the following sets
\begin{equation}
\label{eq:Set_B_new}
\begin{array}{ll}
	\A_\lambda &:= \left\{ x\in X\,:\,(\eco \lambda h)(x)\leq 0\right\}, \text{ for } \lambda\in\RTD_+;\\
	\B&:=\bigcap_{\lambda\in\RTD_+} \A_\lambda.
\end{array}
\end{equation}

\begin{proposition}
\label{prop:Chain_Alamdbda_B}
$\epi\delta_\B^c=\epco\Big(\bigcup_{ \lambda\in\RTD_+}\epi(\lambda h)^c\Big).$

\end{proposition}
\begin{proof}
It is clear that $\eco\lambda h(x)\leq \delta_\B(x)$, for all $x\in X$, and we have
\begin{equation*}
	\epi(\eco\lambda h)^c\subseteq \epi\delta_\B^c, \text{ for all } \lambda\in\RTD_+.
\end{equation*}
Taking into account that $(\eco \lambda h)^c=(\lambda h)^c$, for all $\lambda\in\RTD_+$, according to Corollary \ref{cor:Corollaryconseq}, we have
\begin{equation*}
	\bigcup_{\lambda\in \RTD_+}\epi(\lambda h)^c\subseteq \epi\delta_\B^c. 
\end{equation*}
Taking $\ep$-convex hulls, and knowing that $\epi\delta_\B^c$ is $\ep$-convex, we obtain 
$$\epco\Big(\bigcup_{ \lambda\in\RTD_+}\epi(\lambda h)^c\Big)\subseteq \epi\delta_\B^c .$$\\
Now, it is easy to check that, if $D \subseteq X$, $D \neq \emptyset$, then
\begin{equation}\label{epi_cconj}
\epi \delta_D^c=\bigcap_{x \in D} \epi c(x, \cdot). 
\end{equation}
We define the set
\begin{equation*}
	K:=\bigcup_{\lambda\in\RTD_+}\epi(\lambda h)^c,
\end{equation*}
and let $H:=\sup_{(x,\gamma)\in S} \left\lbrace c(x,\cdot)-\gamma \right\rbrace:W \rightarrow \Ramp, S\subseteq X \times \R$, the $\ep$-convex function such that $\epco K=\epi H$. 
Following the same lines than in the proof of Proposition 4.1 in \cite{FVR2016}, more precisely formulae $(7)-(11)$, it can be shown that
$$\epco K=\bigcap_{x \in X_1} \epi c(x,\cdot),$$
where $X_1$ is the projection of the set $S$ onto $X$. Taking into account \eqref{epi_cconj}, if we see that $X_1 \subseteq \B$, we will have $\epi\delta_\B^c \subseteq \epco K$.\\
Take a point $\bar x \in X_1$ and $\lambda \in \RTD_+$. Since $K \subseteq \bigcap_{x \in X_1} \epi c(x,\cdot)$, it holds
$$\epi (\lambda h)^c \subseteq \epi c(\bar x, \cdot),$$
and
$$c(\bar x, (x^*, y^*, \alpha))-( \lambda h)^c (x^*, y^*, \alpha) \leq 0,$$
for all $(x^*, y^*, \alpha) \in \dom (\lambda h)^c$, then
$$(\lambda h)^{cc'} (\bar x) \leq 0$$ meaning that $(\eco \lambda h)(\bar x) \leq 0$, according to Theorem \ref{thm:Theorem_charac} $ii)$. We conclude that $ \bar x \in \A_\lambda$ and $ X_1 \subseteq \B$.\\
\qed
\end{proof}
\begin{remark}
It is important to make the following observation. As Example \ref{ex:Constraints_need_to_be_econvex} shows, the equality \eqref{eq:Prop4_1_FVR2015} holds provided the constraint functions are e-convex. However, and despite the structure of Proposition \ref{prop:Chain_Alamdbda_B} resembles the same appearance than equation  \eqref{eq:Prop4_1_FVR2015}, that proposition holds independently of the even convexity of the constraints that define the set $A$. Hence, as a final comment in that regard, it is convenient to mention that in case the constraints happen to be e-convex, then
\begin{equation*}
	\epi\delta_A^c=\epi\delta_\B^c
\end{equation*}
and Proposition \ref{prop:Chain_Alamdbda_B} boils down to equation \eqref{eq:Prop4_1_FVR2015}.
\end{remark}

In the light of Proposition \ref{prop:Chain_Alamdbda_B}, we give a new condition in the spirit of (ECCQ), where we allow the constraint functions $h_t, \,t \in T,$ to not be necessarily e-convex.
\begin{definition}
\label{def:ECCQII}
The constraint system $\sigma:=\{  h_t(x)\leq 0, ~t\in T \}$ verifies condition (ECCQII) if $\bigcup_{ \lambda\in\RTD_+}\epi(\lambda h)^c$ is $\ep$-convex.
\end{definition}
\begin{example}
\label{ex:Set_B_after_ECCQII}
Revisiting Example \ref{ex:Constraints_need_to_be_econvex}, we can see that condition (ECCQII) holds. Taking into account that
\begin{equation*}
	\eco \lambda h(x)=\left\{
	\begin{array}{ll}
		\lambda x,&~~~ \text{if }x \leq 0,\\
		+\infty,&~~~ \text{if }x>0,
	\end{array}
	\right.
\end{equation*}
it is clear that $\B=]-\infty,0]$ and $\epi\delta_\B^c=\R_+\times\big(\R_+\times\R_{++}\big)\times\R_+$, so $\epi\delta_A^c\subsetneqq\epi\delta_\B^c$ and, additionally, $A\subsetneqq \B$.\qed
\end{example}

\indent
We state the next result which ensures strong duality for the dual pair \eqref{eq:Primal_problem}$-$\eqref{eq:DL_standard}; see again \cite[Th.~4.1]{FVR2016} for the primal-dual pair $(P_0)-(D_L)$ with e-convex constraints. 
\begin{proposition}
\label{prop:ECCCQ_SD}
If $\sigma$ verifies  \textup{(ECCQII)}, $\mathcal{E}_{f-g}\neq \emptyset$, $\eco(f-g+\delta_\B)=\sup\{a\,:\, a \in \mathcal{\tilde E}_{f-g,\delta_\B}\}$, $\epi(f-g)^c+\epi \delta_\B^c$ is $\ep$-convex and $v(P)=\inf_\B f(x)-g(x)$, then strong duality holds for $(P)-(D_L)$.
\end{proposition}

In the previous proposition we impose the assumption $v(P)=\inf_\B f(x)-g(x)$. Next example shows that this condition does not imply that the set $A$ coincides with $\B$, even when both sets are e-convex.   

\begin{example}
\label{ex:Sets_A_B_v(P)}
Let $T$ be a singleton and $f,\,g:\R\to\R$, $h:\R\to\Ramp$ be functions defined as
\begin{equation*}
	f(x)=x^4,
	g(x)=x^2,
	h(x)=\left\{
	\begin{array}{ll}
		-x,&~~~ \text{if } x> 1,\\
		2,&~~~ \text{if } x= 1,\\
		+\infty, &~~~ \text{otherwise.}
	\end{array}
	\right.
\end{equation*}
We have 
\begin{equation*}
	A = ]1,+\infty[\,\subsetneqq \,[1,+\infty [\ = \B.
\end{equation*}
Now, according to the definition of $f(x)$ and $g(x)$, it is immediate to see that
\begin{equation*}
	v(P)=\inf_{x \in A} f(x)-g(x)=0=\inf_{x \in \B} f(x)-g(x)
\end{equation*}
as we claimed.\qed
\end{example}
\begin{remark}
Example \ref{ex:Sets_A_B_v(P)} shows that despite the fact that the function $g$ can be e-convex, it is not necessary for the equality
\begin{equation*}
	v(P)=\inf_{x \in \B} f(x)-g(x)
\end{equation*}
that the sets $A$ and $\B$ must be equal.
\end{remark}

Next example shows that the sufficient condition (ECCQII) in Proposition \ref{prop:ECCCQ_SD} is not necessary for strong duality for \eqref{eq:Primal_problem}-\eqref{eq:DL_standard} even when all the other hypothesis remain true.
\begin{example}
\label{ex:ECCQ_not_necessary_for_SD}
Let $X=\R$, $T=]-\infty,-1]$ and the functions in the primal problem be as follows
\begin{equation*}
	f(x)=\left\{
	\begin{array}{ll}
		x^3, & ~~~\text{if }x\geq 0,\\
		+\infty, & ~~~\text{otherwise,}
	\end{array}
	\right.
	~~g(x)=x^2,
	~~ h_t(x)=\left\{
	\begin{array}{ll}
		tx, & ~~~\text{if }x>t,\\
		+\infty, & ~~~\text{otherwise,}
	\end{array}
	\right.
	\text{ for all } t\in T.
\end{equation*}
According to the definition of $h_t$, it is immediate to see that $A=[0,+\infty[$ and, since  all the constraints are e-convex, $A=\mathcal{B}$ and (ECCQII) ends up being (ECCQ).\\
\noindent
We begin showing that (ECCQ) does not hold, for which we compute
\begin{align*}
	\epi \delta_A^c &= -\R_+\times -\R_+\times\R_{++}\times\R_+,\\
	\epi (0 h)^c &=\left\lbrace (0,0) \right\rbrace \times \R_{++} \times \R_+, 
\end{align*}
and, in case $\lambda \neq 0$,
\begin{align*}
\epi (\lambda h)^c &= \left\lbrace (x^*, y^*, \alpha, \beta): x^* \leq \sum_{t \in T} \lambda_t t, y^* \leq 0,\alpha >\hat{t} y^*, \beta \geq \hat{t}(x^*-\sum_{t \in T} \lambda_t t)\right\rbrace,
\end{align*}
being $\hat{t}=\max\left\{t,\,\lambda_t\neq 0\right\}\leq -1$.
Hence
\begin{equation}\label{K}
\bigcup_{\lambda\in\RTD_+} \epi (\lambda h)^c = (-\R_{++}\times -\R_{+}) \cup \left\lbrace (0,0) \right\rbrace \times \R_{++} \times \R_+.
\end{equation}
Finally,
\begin{equation*}
	\bigcup_{\lambda\in\RTD_+} \epi (\lambda h)^c \nsubseteqq \epi \delta_A^c,
\end{equation*}
and, consequently, (ECCQ) does not hold according to Proposition \ref{prop:Chain_Alamdbda_B}.\\
\noindent
Following the list of hypothesis in Proposition \ref{prop:ECCCQ_SD}, we continue studying the function
\begin{equation*}
	(f-g)(x)=\left\{
	\begin{array}{ll}
		x^3-x^2,&~~~\text{if } x\geq 0,\\
		+\infty, &~~~\text{otherwise.}
	\end{array}
	\right.
\end{equation*}
Computing the tangent lines to $f-g$ which minorize it, we obtain that its $c$-elementary minorant functions are
\begin{equation}\label{celementary}
a(x)=\left\{
	\begin{array}{ll}
		xx^* -\beta, & ~~~\text{if }xy^*<\alpha,\\
		+\infty, & ~~~\text{otherwise,}
	\end{array}
	\right.
\end{equation}
where either $(x^*,\beta) \in \left\lbrace (3a^2-2a,-2a^3+a^2+\gamma):a \geq {1 \over 2}, \gamma \geq 0 \right\rbrace$ (in this case, $x^* \geq -{1 \over 4}$) or  $(x^*,\beta) \in \left]-\infty,-{1 \over 4} \right[\times \R_+$  and $(y^*, \alpha) \in -\R_+ \times \R_{++}$.
 Then $\mathcal{E}_{f-g}\neq\emptyset$. On the other hand, the $c$-elementary minorant functions of $\delta_A$ fulfills, following the notation in (\ref{celementary}), that
 $(x^*,\beta,y^*, \alpha) \in -\R_+ \times \R_{+} \times -\R_{+} \times\R_{++}$ and clearly we obtain $\mathcal{E}_{f-g} \subseteq \widetilde{\mathcal{E}}_{f-g,\delta_A}$. Let us check the converse inclusion.
Take $a \in \widetilde{\mathcal{E}}_{f-g,\delta_A}$. Then we can find $a_1 \in  \mathcal{E}_{f-g}$ and $a_2 \in  \mathcal{E}_{\delta_A}$ such that, if we identify $a \equiv (x^*,\beta,y^*, \alpha),\, a_i \equiv (x_i^*,\beta_i,y_i^*, \alpha_i)$, for $i=1,2$, then $x^*=x_1^*+x_2^*,\, \beta=\beta_1+\beta_2,\, y^*=y_1^*+y_2^*$ and 
$\alpha=\alpha_1+\alpha_2$. Since it is immediate in the case $x^* < -{1 \over 4}$, let us assume that $x^* \geq -{1 \over 4}$, which means that necessarily $x_1^* \geq -{1 \over 4}$. Let $a,a_1 \geq {1 \over 2}$ such that $x^*=3a^2-2a$ and $x_1^*=3a_1^2-2a_1$, with $\beta_1 \geq -2a_1^3+a_1^2$. Name $\hat{\beta}:= -2a^3+a^2$. Then $\beta \geq \beta_1 \geq \hat{\beta}$, since $x_1^* \geq x^*$. 
Clearly $y^* \leq 0$, $\alpha >0$ and $ a \in \mathcal{E}_{f-g}$.
 
 Taking into account that 
\begin{equation*}
	\eco(f-g+\delta_A)=\left\{
	\begin{array}{ll}
		+\infty,&~~~ \text{if }x<0,\\
		-\frac{1}{4}x,&~~~\text{if } 0\leq x\leq \frac{1}{2},\\
		x^3-x^2,&~~~ \text{if } x>\frac{1}{2},
	\end{array}
	\right.
\end{equation*}
it holds
\begin{equation*}
	\sup\left\{a\,:\,a\in\mathcal{E}_{f-g}\right\} = \eco (f-g+\delta_A).
\end{equation*}
\noindent
Next item to check is the $\ep$-convexity of the set $\epi(f-g)^c+\epi \delta_A^c$. By the definition of $c$-conjugate function, we have
\begin{equation*}
	(f-g)^c(x^*,y^*,\alpha)=\left\{
	\begin{array}{ll}
		\sup_{x\geq 0} \left\{-x^3+x^2+x x^*\right\},&~~~\text{if } xy^*<\alpha,\\
		+\infty,&~~~\text{otherwise,}
	\end{array}
	\right.
\end{equation*}
which, after its evaluation we get
\begin{equation*}
	(f-g)^c(x^*,y^*,\alpha)=\left\{
	\begin{array}{ll}
		0,&~~~\text{if } x^*\leq -1/4,\,y^*\leq 0,\, \alpha>0,\\
		\beta(x^*),&~~~\text{if } x^*> -1/4,\,y^*\leq 0,\, \alpha>0,\\
		+\infty,&~~~\text{otherwise,}
	\end{array}
	\right.
\end{equation*}
with $\beta(x^*)=-(a^*)^3+(a^*)^2+x^*a^*$ being $a^*=\frac{1+\sqrt{1+3x^*}}{3}$. Hence,
\begin{align*}
	\epi(f-g)^c=&\big( ]-\infty,-1/4]\times-\R_+\times \R_{++}\times\R_+\big) \,\bigcup\\
	&\left\{ (x^*,y^*,\alpha,\beta)\,:\, x^*>-1/4,\,y^*\leq 0,\,\alpha>0,\,\beta\geq\beta(x^*)\right\},
\end{align*}
and, due to the expression of $\epi\delta_A^c$ previously calculated, 
\begin{equation*}
	\epi(f-g)^c+\epi\delta_A^c=\epi(f-g)^c,
\end{equation*}
which is $\ep$-convex.\\
\noindent
Now, since $\mathcal{B}=A$, 
\begin{equation*}
	v(P)=\inf_{x \in \B}(f-g+\delta_A).
\end{equation*}
 \noindent
Last item to verify is the fulfillment of strong duality between \eqref{eq:Primal_problem} and \eqref{eq:DL_standard}. Using previous calculations
 in this example, it follows
\begin{equation*}
	v(D_L)
	= \sup_{\lambda\in\RTD_+}\inf_{x\geq 0} \left\{ f(x)-g(x)+\lambda h(x)\right\}
	= \left(\frac{2}{3}\right)^3-\left(\frac{2}{3}\right)^2
	=\inf_{x\geq 0} \left\{x^3-x^2\right\} 
	= v(P),
\end{equation*}
being the dual problem solvable, which finally shows that (ECCQ) is not necessary for strong duality even in DC optimization problems.\qed
\end{example}

As an insight on the hypothesis stated in Proposition \ref{prop:ECCCQ_SD}, the following result relates, whenever $g$ is e-convex, the set $\mathcal{E}_{f-g}$ to the sets $\mathcal{E}_{f}$ and $\mathcal{E}_{g}$, and establishes a formula for $\epi(f-g)^c$. 
\begin{proposition}
\label{prop:More_results_on_E}
Let $g$ be an e-convex function. The following statements are true
\begin{itemize}
	\item [i)]$\mathcal{E}_{f-g}\subseteq\mathcal{E}_{f}-\mathcal{E}_{g};$
	\item[ii)] $\epi(f-g)^c
	=
	\bigcap_{(u^*,v^*,\gamma)\in\dom g^c} \left\{ \epi f^c-(u^*,0,0,g^c(u^*,v^*,\gamma))\right\};$
	\item[iii)] $\epco \big \lbrace {\epi(f-g)^c+\epi \delta_\B^c\big \rbrace}=
	\epi \inf \{h^c: h\in \mathcal{\tilde E}_{f-g, \delta_\B}\}.$
\end{itemize}
\end{proposition}
\begin{proof}
$i)$ Take an e-affine function $a \in \mathcal{E}_{f-g}$. Then, for all $x \in X$, $$a(x)+g(x)\leq f(x).$$ 
According to Theorem \ref{theorem:charpef}, $g=\sup\{b:b\in \mathcal{E}_{g}\}$, so naming $\mathcal{A}:=\{b\in \mathcal{E}_{g}:a+b \text{ is e-affine}\}$, we obtain, for all $x \in X$,
\begin{equation*}
	\sup_{b\in \mathcal{A}} (a+b)(x)\leq f(x)
\end{equation*}
and 
\begin{equation*}
	\sup_{b\in \mathcal{A}} (a+b)(x)\leq \sup_{c \in \mathcal{E}_{f}}c(x)\leq f(x).
\end{equation*}
Therefore, $a+b \in \mathcal{E}_{f}$ and $a \in \mathcal{E}_{f}-\mathcal{E}_g$. \\\\
$ii)$ Due to basic properties of the conjugate function and the e-convexity of the function $g$, we have
\begin{align*}
	(f- g)^c &= \biggl(f-\sup_{(u^*,v^*,\gamma)\in\dom g^c}\left\{ c(\cdot,(u^*,v^*,\gamma))-g^c(u^*,v^*,\gamma)\right\}\biggr)^c\\
	&= \left(\inf_{(u^*,v^*,\gamma)\in\dom g^c} \biggl\{ f+g^c(u^*,v^*,\gamma)-c(\cdot,(u^*,v^*,\gamma))\biggr\} \right)^c\\
	&= \sup_{(u^*,v^*,\gamma)\in\dom g^c} \biggl\{ f-c(\cdot,(u^*,v^*,\gamma))+g^c(u^*,v^*,\gamma)\biggr\}^c.
\end{align*}
Therefore, we can derive
\begin{equation*}
	\epi(f- g)^c = \bigcap_{(u^*,v^*,\gamma)\in\dom g^c} \epi\Bigl(f-c(\cdot,(u^*,v^*,\gamma))+g^c(u^*,v^*,\gamma)\Bigr)^c.
\end{equation*}
So, if $(x^*,y^*,\alpha,\beta)\in \epi(f-g)^c$, it is equivalent to 
\begin{equation*}
	(x^*,y^*,\alpha,\beta)\in\epi\Bigl(f-c(\cdot,(u^*,v^*,\gamma)) + g^c(u^*,v^*,\gamma)	\Bigr)^c,
\end{equation*}
for any point $(u^*,v^*,\gamma) \in \dom g^c$, and using the fact that $\dom f \subseteq \dom g \subseteq H_{v^*,\gamma}^-$, we conclude that $f^c(x^*+u^*,y^*,\alpha)\leq \beta+g^c(u^*,v^*,\gamma)$,
which means that
\begin{equation*}
	(x^*+u^*,y^*,\alpha,\beta+g^c(u^*,v^*,\gamma)) \in \epi f^c
\end{equation*}
or, equivalently, $(x^*,y^*,\alpha,\beta)\in \epi f^c-(u^*,0,0,g^c(u^*,v^*,\gamma)).$\\\\
$iii)$ Applying conveniently \cite[Th.~15]{FVR2012}, we conclude that
\begin{equation*}
	\epco \big \lbrace {\epi(f-g)^c+\epi \delta_\B ^c\big \rbrace}=\epi \left( \sup \big \{h:h\in \mathcal{\tilde E}_{f-g, \delta_\B}\big \}\right)^c.
\end{equation*}
The desired results comes directly from the equality
\begin{equation*}
	 \Big( \sup \big \{h:h \in \mathcal{H}\big \}\Big)^c=\inf \{h^c: h \in \mathcal{H}\},
\end{equation*}
where $\mathcal{H}$ is an arbitrary family of functions.
\qed
\end{proof}
\begin{remark}
Observe that the set $\mathcal{A}$ defined in the proof of Proposition \ref{prop:More_results_on_E} $i)$ is always nonempty. The reason for this is because, firstly, any affine function $\left\langle \cdot, x^* \right\rangle-\beta$ is trivially an e-affine function (take $u^*=0$ and $\alpha >0$). Secondly, for any proper and convex function there always exist affine minorants and, finally, the addition of an e-affine and an affine functions is an e-affine function.
\end{remark}

\section{Duality results for \eqref{eq:Primal_problem}$-$\eqref{eq:DL_tilde}}
\label{sec:SD_DL_tilde}

The purpose of this section is mainly to establish conditions ensuring strong duality for the dual pair \eqref{eq:Primal_problem}$-$\eqref{eq:DL_tilde}. Nevertheless and for the sake of completeness, we previously give an insight concerning weak duality for this primal-dual pair.

\subsection{Weak duality}
We would like to start this subsection mentioning that the fulfillment of condition (AC), for any $\lambda\in\RTD_+$,
\begin{equation}
	\label{eq:AC_flambdah}
	\eco (f+ \lambda h) =\sup \{ a: a \in \mathcal{\tilde E}_{f,\lambda h}\}
\end{equation}
is the unique responsible for the inequality $v\eqref{eq:DL_bar}\geq v\eqref{eq:DL_tilde}$. In other words, the function $g$ plays no role at all between the comparison of the optimal values of these two dual problems. For this reason, we recap that every result dealing with weak duality between \eqref{eq:Primal_problem}$-$\eqref{eq:DL_bar} treated in \cite[Sect.~4]{FV2024}, also applies here ensuring weak duality for the dual pair \eqref{eq:Primal_problem}$-$\eqref{eq:DL_tilde}. 

However, since $v\eqref{eq:DL_tilde}$ lies below $v\eqref{eq:DL_bar}$ provided \eqref{eq:AC_flambdah} holds, those regularity conditions for the latter will not ensure, necessarily, strong duality for the former and we need to go into their analysis in the next subsection.

\subsection{Strong duality} 
Let us consider again the set defined in the proof of Proposition \ref{prop:Chain_Alamdbda_B}
\begin{equation}\label{eq:Definition_of_K_for_us}
	K=
	\bigcup_{\lambda\in\RTD_+} \epi(\lambda h)^c 
\end{equation}
and set the condition
\begin{equation}
	\tag{$\ECC$}
	\epi f^c + K ~~\text{is } \ep\text{-convex}.
\end{equation} 
In order to put some context regarding this set and condition, it is deserved to be mentioned that condition (ECC) was used in \cite{FV2018} seeking for strong Fenchel-Lagrange duality conditions for optimization problems $(P_0)$, and $f$ and $h_t$, for all $t\in T$, were all proper e-convex functions. On the other hand, the condition for $K$ to be $\ep$-convex is exactly the fulfillment of \eqref{eq:ECCQ}, since $\epco K=\epi \delta_A^c$ due to \eqref{eq:Prop4_1_FVR2015} in that context.

Now, without any requirement of even convexity of the functions involved in the problem $(P)$, we state the main result in this section, which, apart from being a crucial link between the (ECC) property and strong duality for $(P)-\eqref{eq:DL_tilde}$, it also provides a characterization of (ECC) in terms of  $\varepsilon$-$c$-subdifferentials; see \cite[Th.~3.1]{DNV2010} for its classical counterpart using standard duality theory and also \cite[Th.~11]{FVR2012} for a related result in the line of evenly convex optimization. 

\begin{theorem}\label{thm:Theorem3.1}
If $f$ and $\delta_\B$ satisfy condition \emph{(AC)} and $K$ is $\ep$-convex, the following statements are equivalent:
\begin{itemize}
	\item[i)] Condition \emph{(ECC)} holds;
	
	\item[ii)] For every $(x^*,y^*,\alpha)\in W$,
	\begin{equation*}
	\begin{array}{ll}
		&(f+\delta_\B)^c(x^*,y^*,\alpha)\\
		&=\min_{\lambda\in\RTD_+} \min_{(u^*,v^*,\gamma)\in W} \left\{f^c(u^*,v^*,\gamma)+(\lambda h)^c(x^*-u^*,y^*-v^*,\alpha-\gamma)\right\}.
	\end{array}
	\end{equation*}
	\item[iii)] For all $\xb\in \B\cap \dom f$ and $\varepsilon\geq 0$,
	\begin{equation}\label{eq:Eq_6_Thm_3.1}
		\partial_{c,\varepsilon} (f+\delta_\B)(\xb)
		=
		\bigcup_{\lambda\in\RTD_+} \bigcup_{\substack{\varepsilon_1\geq 0,\varepsilon_2\geq 0\\\varepsilon_1+\varepsilon_2=\varepsilon+(\eco\lambda h)(\xb)}} ~ \left\{ \partial_{c,\varepsilon_1} f(\xb) + \partial_{c,\varepsilon_2} (\eco \lambda h)(\xb)\right\}.
	\end{equation}
\end{itemize}
\end{theorem}
\begin{proof}
\textit{i) $\Rightarrow$ ii)} Take $(x^*,y^*,\alpha)\in W$. Then, for any $(u^*,v^*,\gamma)\in W$, $\lambda\in \RTD_+$ and $x\in \B$, we have, taking into account Corollary \ref{cor:Corollaryconseq},
\begin{align*}
	& f^c(u^*,v^*,\gamma)+(\lambda h)^c(x^*-u^*,y^*-v^*,\alpha-\gamma)\\
	&= f^c(u^*,v^*,\gamma)+(\eco\lambda h)^c(x^*-u^*,y^*-v^*,\alpha-\gamma)\\
	& \geq c(x,(x^*,y^*,\alpha))-f(x),
\end{align*}
thanks to the definition of the $c$-conjugate function, the subadditivity of the coupling function in its second component and the fact that $(\eco\lambda h)(x)\leq 0$ provided $x\in \B$. This implies that
\begin{align*}
	f^c(u^*,v^*,\gamma)+(\lambda h)^c(x^*-u^*,y^*-v^*,\alpha-\gamma)\geq c(x,(x^*,y^*,\alpha))-(f+\delta_\B)(x),
\end{align*}
for all $(u^*,v^*,\gamma)\in W$, $\lambda\in \RTD_+$ and $x\in X$, obtaining, for each $(x^*,y^*,\alpha) \in W$,
\begin{equation}\label{eq:8_Thm_3.1}
\begin{array}{ll}
	& \inf_{\lambda\in\RTD_+} \inf_{(u^*,v^*,\gamma)\in W} ~ f^c(u^*,v^*,\gamma)+(\lambda h)^c(x^*-u^*,y^*-v^*,\alpha-\gamma)\\
& \geq ~ \sup_{x\in X} \{c(x,(x^*,y^*,\alpha))-(f+\delta_\B)(x)\}\\
& = (f+\delta_\B)^c(x^*,y^*,\alpha).
\end{array}
\end{equation}
We will see that both infima at (\ref{eq:8_Thm_3.1}) are attained. Clearly, if $(x^*,y^*,\alpha)\notin \dom (f+\delta_\B)^c$, $(f+\delta_\B)^c(x^*,y^*,\alpha)=+\infty$ and due to (\ref{eq:8_Thm_3.1}) we obtain \textit{ii)}. In other case, thanks to condition (AC), Theorem 4 in \cite{FVR2012} can be applied, and we obtain, taking into account Corollary \ref{cor:Corollaryconseq}, the following chain of equalities
\begin{equation}\label{eq:Th4.13}
	\epi (f+\delta_\B)^c = \epi \left(\eco(f+\delta_\B)\right)^c= \epco (\epi f^c+\epi \delta_\B^c).
	\end{equation}
Now recalling Proposition \ref{prop:Chain_Alamdbda_B}, the $\ep$-convexity of $K$  allows us to write	
	$$\epi (f+\delta_\B)^c = \epco \left( \epi f^c + \bigcup_{\lambda \in \RTD_+} \epi (\lambda h)^c \right),$$
which, together with condition (ECC), it is equivalent to the equality	
$$\epi (f+\delta_\B)^c=  \epi f^c + \bigcup_{\lambda  \in \RTD_+} \epi (\lambda h)^c. $$
Then, $(x^*,y^*,\alpha,(f+\delta_\B)^c(x^*,y^*,\alpha))\in \epi f^c + \bigcup_\lambda \epi (\lambda h)^c$, and there exist $\bar\lambda\in\RTD_+$ and $(\bar u^*,\bar v^*,\bar \gamma, \bar\beta)\in \epi f^c$, such that
\begin{equation*}
	\Big(x^*-\bar u^*,y^*-\bar v^*,\alpha-\bar \gamma,(f+\delta_\B)^c(x^*,y^*,\alpha)- \bar \beta\Big) \in \epi (\bar \lambda h)^c.
\end{equation*}
Hence
\begin{align*}
	(f+\delta_\B)^c(x^*,y^*,\alpha) &=  \bar\beta +(f+\delta_\B)^c(x^*,y^*,\alpha)- \bar \beta\\
	&\geq f^c(\bar u^*,\bar v^*,\bar \gamma) + (\bar \lambda h)^c(x^*-\bar u^*,y^*-\bar v^*,\alpha-\bar \gamma).
\end{align*}
Due to (\ref{eq:8_Thm_3.1}) both infima are attained and, moreover, equality in $ii)$ holds.\\

\textit{ii) $\Rightarrow$ iii)} Let $\varepsilon\geq 0$, $\xb\in\dom(f+\delta_\B)$ and take $(x^*,y^*,\alpha)$ a point from the right-hand-side in (\ref{eq:Eq_6_Thm_3.1}). Let $\lambda\in\RTD_+$, $\varepsilon_i\geq 0$ with $i=1,2$, $\varepsilon_1+\varepsilon_2=\varepsilon+(\eco \lambda h)(\xb)$, verifying
\begin{equation*}
	\left.
	\begin{array}{rl}
		(u^*,v^*,\gamma) & \in\partial_{c,\varepsilon_1} f(\xb),\\	
		(x^*-u^*,y^*-v^*,\alpha-\gamma) & \in\partial_{c,\varepsilon_2} (\eco\lambda h)(\xb).
	\end{array}
	\right\}
\end{equation*}
Applying the definition of the $c$-subdifferential set we obtain
\begin{align*}
	 &(f + \eco\lambda h)(x) -(f+\eco\lambda h)(\xb)\\
	 & \geq c(x,(u^*,v^*,\gamma))+c(x,(x^*-u^*,y^*-v^*,\alpha-\gamma))\\
	 &- \left[c(\bar x,(u^*,v^*,\gamma))+c(\bar x,(x^*-u^*,y^*-v^*,\alpha-\gamma))\right]- (\varepsilon_1+\varepsilon_2),
\end{align*}
with $\ci \xb,v^*\cd < \gamma$ and $\ci \xb,y^*-v^*\cd < \alpha-\gamma$ . Thanks to the subadditivity of the coupling function in its second component, and taking into account that $\varepsilon_1+\varepsilon_2=\varepsilon+(\eco\lambda h)(\xb)$ and $\ci \xb,y^*\cd < \alpha$, it holds
\begin{equation*}
	f(x)-f(\xb)+(\eco\lambda h)(x) \geq
	c(x,(x^*,y^*,\alpha)) -c(\bar x,(x^*,y^*,\alpha)) -\varepsilon, 
	\end{equation*}
for all $x\in X$. According to the definition of the set $\B$, we can rewrite this inequality as follows
\begin{equation*}
	f(x)-f(\xb) \geq
	c(x,(x^*,y^*,\alpha)) -c(\bar x,(x^*,y^*,\alpha)) -\varepsilon,
\end{equation*}
for all $x\in \B$, and $\ci \xb,v^*\cd<\alpha$. Equivalently,
\begin{equation*}
	(f+\delta_\B)(x)-(f+\delta_\B)(\xb) \geq
	c(x,(x^*,y^*,\alpha)) -c(\bar x,(x^*,y^*,\alpha)) -\varepsilon, 
\end{equation*}
for all $x\in X$, and $\ci \xb,y^*\cd<\alpha$, hence
\begin{equation*}
	(x^*,y^*,\alpha)\in \partial_{c,\varepsilon}(f+\delta_\B)(\xb).
\end{equation*}
Let us prove the converse inclusion in the case $\partial_{c,\varepsilon} (f+\delta_\B)(\xb)\neq \emptyset$ (otherwise it is trivial). Take $(x^*,y^*,\alpha)\in\partial_{c,\varepsilon}(f+\delta_\B)(\xb)$. Since $\xb\in\dom(f+\delta_\B)$, from Lemma \ref{Lema9_9} we obtain
\begin{equation}\label{eq:11_Thm_3.1}
	(f+\delta_\B)^c(x^*,y^*,\alpha) \leq \varepsilon + \ci \xb,x^*\cd-(f+\delta_\B)(\xb).
\end{equation}
Assuming that \textit{ii)} is true, there exist $\lambda \in{\RTD_+}$ and $(u^*,v^*,\gamma)\in W$ such that
\begin{equation*}
	(f+\delta_\B)^c(x^*,y^*,\alpha)=f^c(u^*,v^*,\gamma)+(\lambda h)^c(x^*-u^*,y^*-v^*,\alpha-\gamma)
\end{equation*}
\noindent
and, rewriting (\ref{eq:11_Thm_3.1}), it follows
\begin{equation}\label{eq:12_Thm3.1}
	f^c(u^*,v^*,\gamma)+(\lambda h)^c(x^*-u^*,y^*-v^*,\alpha-\gamma) \leq
	\varepsilon + \ci \xb,x^*\cd-(f+\delta_\B)(\xb),
\end{equation}
which leads $(u^*,v^*,\gamma)\in \dom f^c$ and $(u^*,v^*,\gamma,f^c(u^*,v^*,\gamma))\in \epi f^c$. Applying Lemma \ref{Lema9_9} once more, there exists $\varepsilon_1 \geq 0$ such that $(u^*,v^*,\gamma)\in\partial_{c,\varepsilon_1} f(\xb)$ and
\begin{equation}\label{eq:a_Thm3.1}
	 f^c(u^*,v^*,\gamma)=\ci \xb,u^*\cd + \varepsilon_1-f(\xb).
\end{equation}

\noindent
Similarly, 
\begin{equation*}
	\big(x^*-u^*,y^*-v^*,\alpha-\gamma,(\lambda h)^c(x^*-u^*,y^*-v^*,\alpha-\gamma)\big)\in \epi (\lambda h)^c=\epi (\eco \lambda h)^c,
\end{equation*}
so we can find $\varepsilon'_2 \geq 0$ verifying $(x^*-u^*,y^*-v^*,\alpha-\gamma)\in\partial_{c,\varepsilon' _2} (\eco\lambda h)(\xb)$ and
\begin{equation}\label{eq:b_Thm3.1}
	(\lambda h)^c(x^*-u^*,y^*-v^*,\alpha-\gamma)=\ci \xb,x^*-u^*\cd + \varepsilon' _2-(\eco\lambda h)(\xb).
\end{equation}
Replacing $f^c(u^*,v^*,\gamma)$ and $(\lambda h)^c(x^*-u^*,y^*-v^*,\alpha-\gamma)$ in (\ref{eq:12_Thm3.1}) from the equalities (\ref{eq:a_Thm3.1}) and (\ref{eq:b_Thm3.1}), respectively, we have that
\begin{equation*}
	\varepsilon+(\eco\lambda h)(\xb) \geq \varepsilon_1+\varepsilon' _2.
\end{equation*}
Defining $\varepsilon_2:=\varepsilon+(\eco\lambda h)(\xb) -\varepsilon_1$, it holds $\varepsilon+(\eco \lambda h)(\xb)= \varepsilon_1+\varepsilon_2$ and, since $\varepsilon_2\geq \varepsilon^\prime_2$, clearly $(x^*-u^*,y^*-v^*,\alpha-\gamma)\in\partial_{c,\varepsilon _2} (\eco\lambda h)(\xb)$.\\

\textit{iii) $\Rightarrow$ i)} Recalling that equation \eqref{eq:Th4.13} is true under the assumptions imposed in this theorem,  take $(x^*,y^*,\alpha,\beta)\in\epco(\epi f^c+K)=\epi(f+\delta_\B)^c$. Fix a point $\xb\in \dom(f+\delta_\B)$ and apply Lemma \ref{Lema9_9}, then there exists $\varepsilon\geq 0$ such that
\begin{align*}
	(x^*,y^*,\alpha)&\in\partial_{c,\varepsilon}(f+\delta_\B)(\xb),\\
	\beta&=\ci \xb,x^*\cd-f(\xb)+\varepsilon.
\end{align*}
Then, from \textit{iii)}, there exist $\lambda\in \RTD_+$, $(u^*,v^*,\gamma)\in W$ and $\varepsilon_1, \varepsilon_2\geq 0$, such that
\begin{align*}
	(u^*,v^*,\gamma)&\in\partial_{c,\varepsilon_1} f(\xb),\\
	(x^*-u^*,y^*-v^*,\alpha-\gamma)&\in\partial_{c,\varepsilon_2} (\eco \lambda h)(\xb),\\
	\varepsilon_1+\varepsilon_2&=\varepsilon+(\eco \lambda h)(\xb).
\end{align*}
If we define
\begin{align*}
	s&:=\ci\xb,u^*\cd-f(\xb)+\varepsilon_1,\\
	t&:=\ci\xb,x^*-u^*\cd-(\eco\lambda h)(\xb)+\varepsilon_2,
\end{align*}
and apply again Lemma \ref{Lema9_9} together with Corollary \ref{cor:Corollaryconseq}, we obtain
\begin{align*}
	(u^*,v^*,\gamma,s)&\in\epi f^c,\\
	(x^*-u^*,y^*-v^*,\alpha-\gamma,t)&\in\epi (\lambda h)^c.
\end{align*}
Moreover, $s+t=\beta$ and $\varepsilon_1+\varepsilon_2=\varepsilon+(\eco\lambda h)(\xb)$. Hence, $(x^*,y^*,\alpha,\beta)\in \epi f^c +\epi(\lambda h)^c$, the following inclusion holds
\begin{equation*}
	\epco(\epi f^c+K) \subseteq \epi f^c + K
\end{equation*}
and the set $\epi f^c +  K$ is $\ep$-convex. \qed
\end{proof}

Next result is useful in the development of a new regularity condition for strong duality between \eqref{eq:Primal_problem} and \eqref{eq:DL_tilde}. It follows from \cite[Th.~3.1]{ML1990} adding that $g$ is an e-convex function.
\begin{lemma}\label{cor:TD_g_econvex}
Let $f,g:X\to\Ramp$ be proper functions with $g$ e-convex. Then,
\begin{equation*}
	\inf_{x\in X} \left\{ f(x)-g(x)\right\}
	=
	\inf_{(x^*,y^*,\alpha)\in W} \left\{ g^c(x^*,y^*,\alpha) - f^c(x^*,y^*,\alpha)\right\}.
\end{equation*}
\end{lemma}

\begin{theorem}\label{thm:RC_for_SD_P_DLTilde}
Suppose that $f$ and $\delta_\B$ satisfy condition (AC), $g$ is an e-convex function, $K$ is an $\ep$-convex set, (ECC) holds and $v(P)=\inf_{\B}f(x)-g(x)$. Then
\begin{equation}\label{eq:TFL_duality}
	\inf_{x\in A} \{f(x)-g(x)\}
	=
	 \max_{\lambda\in\RTD_+} \inf_{(x^*,y^*,\alpha)\in W} \{g^c(x^*,y^*,\alpha) - (f^c\oplus (\lambda h)^c)(x^*,y^*,\alpha)\}
	 \end{equation}
	 and strong duality holds for \eqref{eq:Primal_problem}$-$ \eqref{eq:DL_tilde}.
\end{theorem}
\begin{proof}
We start rewriting the primal problem as follows
\begin{equation*}
	\inf_{x\in A} \{f(x)-g(x)\}=\inf_{x\in X} \left\{(f+\delta_\B)(x)-g(x)\right\}.
\end{equation*}
By virtue of Lemma \ref{cor:TD_g_econvex}, the following equality holds
\begin{equation}
	\label{eq:SDTilde_F1}
	\inf_{x\in X} \left\{(f+\delta_\B)(x)-g(x)\right\}= \inf_{(x^*,y^*,\alpha)\in W} \left\{g^c(x^*,y^*,\alpha)-(f+\delta_\B)^c(x^*,y^*,\alpha)\right\},
\end{equation}
since $g$ is assumed to be e-convex.
Due to Theorem \ref{thm:Theorem3.1} item $ii)$,
\begin{align}
	&(f+\delta_\B)^c(x^*,y^*,\alpha)\nonumber\\
	&=\min_{\lambda\in\RTD_+}\min_{(u^*,v^*,\gamma)\in W} f^c(u^*,v^*,\gamma)+(\lambda h)^c(x^*-u^*,y^*-v^*,\alpha-\gamma)\nonumber\\
	&=\min_{\lambda\in\RTD_+}\left(f^c\oplus(\lambda h)^c\right)(x^*,y^*,\alpha).\label{eq:SDTilde_F2}
\end{align}
To conclude the proof, it is enough to replace \eqref{eq:SDTilde_F2} in \eqref{eq:SDTilde_F1} to finally obtain \eqref{eq:TFL_duality}. \qed
\end{proof}

Next example shows that actually condition (ECCQII), the $\ep$-convexity of $K$, a sufficient condition for the characterization of (ECC) according to Theorem \ref{thm:RC_for_SD_P_DLTilde}, is not necessary. 
\begin{example}
Coming back to Example \ref{ex:ECCQ_not_necessary_for_SD}, we have
\begin{align*}
	\epi f^c=&\big( -\R_+ \times -\R_+ \times \R_{++} \times \R_+ \big)\bigcup\\ 
	&\left\{ (x^*, y^*, \alpha, \beta): x^*>0, y^* \leq 0, \alpha >0, \beta \geq 2 \left(\frac{x^*}{3}\right)^{3/2} \right\},
\end{align*}
whereas, according to (\ref{K}),
\begin{equation*}
K=\left\lbrace-\R_{++}\times -\R_{+}\right\rbrace \cup \left\lbrace (0,0) \right\rbrace \times \R_{++} \times \R_+.
\end{equation*}
Hence,
\begin{equation*}
	\epi f^c+K=\epi f^c
\end{equation*}
and (ECC) holds, despite $K$, as we have seen in Example \ref{ex:ECCQ_not_necessary_for_SD}, is not $\ep$-convex.
It is not difficut to check that condition (AC) holds for $f$ and $\delta_A$.
On one hand, $\eco(f+\delta_A)=f$, since $f$ is e-convex. Due to Theorem \ref{theorem:charpef}, condition (AC) will fulfill if $\widetilde{\mathcal{E}}_{f,\delta_A}=\mathcal{E}_f.$ We have, following the notation in (\ref{celementary}), that the elements characterizing any $c$-elementary minorant of $f$ must verify
either $(x^*,\beta) \in \left\lbrace (3a^2,2a^3+\gamma):a >0, \gamma \geq 0 \right\rbrace$ or  $(x^*,\beta) \in - \R_+ \times \R_+$  and $(y^*, \alpha) \in -\R_+ \times \R_{++}$.
The set of $c$-elementary minorants of $\delta_A$ was calculated in Example \ref{ex:ECCQ_not_necessary_for_SD}, and clearly $\mathcal{E}_f \subseteq \widetilde{\mathcal{E}}_{f,\delta_A}$. Let us show the converse inclusion.
Take $a \in \widetilde{\mathcal{E}}_{f,\delta_A}$, $a_1 \in  \mathcal{E}_f$ and $a_2 \in  \mathcal{E}_{\delta_A}$ such that, if we identify $a \equiv (x^*,\beta,y^*, \alpha),\, a_i \equiv (x_i^*,\beta_i,y_i^*, \alpha_i)$, for $i=1,2$, then $x^*=x_1^*+x_2^*,\, \beta=\beta_1+\beta_2,\, y^*=y_1^*+y_2^*$ and 
$\alpha=\alpha_1+\alpha_2$. Since it is immediate in the case $x^* \leq 0$, let us assume that $x^* >0$ and, consequently, $x_1^* >0$ and $\beta_1 \geq 2 \left(x_1^* \over 2 \right)^{3/2}$. Now, since $x_2^* \leq 0$ and $\beta_2 \geq 0$, it holds
$$\beta \geq \beta_1 \geq 2 \left(x_1^* \over 2 \right)^{3/2} \geq 2 \left(x^* \over 2 \right)^{3/2}.  $$
Clearly $y^* \leq 0$, $\alpha >0$ and $ a \in \mathcal{E}_f$.
\end{example}

\begin{remark}
It deserves to be mentioned that Theorem \ref{thm:RC_for_SD_P_DLTilde} is also referred in the literature as \textit{Toland-Fenchel-Lagrange} duality; see for instance \cite{DNV2010}.
\end{remark}

\section{Conclusions}
\label{sec:Conclusions}

This paper has covered two different purposes. On the one hand and motivated by \cite{FVR2016}, we recover and generalize a regularity condition for strong Lagrange duality supposing the involved functions in the primal to be just convex instead of e-convex. As a second goal, we develop a new Lagrange dual problem using the infimal convolution operator and going further with the expression of a previous Lagrange dual obtained in \cite{FV2024}. Once we state it, we characterize a constraint qualification (ECC) using $\varepsilon$-$c$-subdifferentials and prove that it ensures strong Lagrange duality between the primal DC problem and its new Lagrange in terms of the infimal convolution.

As future research, we mention that in the same way that our condition (ECC) turned out to be sufficient for strong Fenchel-Lagrange duality for the case of the primal problem $(P_0)$, see \cite[Prop.~3.4]{FV2018}, it may be useful to continue analyzing its relation with duality results involving the Fenchel-Lagrange dual problem of a (DC) primal recently obtained in \cite{FV2024}.

\section*{Funding}
Research partially supported by Grant PID2022-136399NB-C21 funded by MICIU/AEI/10.13039/501100011033 and by ERDF/EU.

\section*{Conflict of interest}
The authors declare that they have no conflict of interest.

\bibliographystyle{plain}
\bibliography{biblio.bib}

\end{document}